\documentclass{amsart}

\usepackage{amsmath}
\usepackage{amsthm}
\usepackage{amsfonts}
\usepackage{amssymb, amscd}

\newcommand{\bracketa}[1]{\big[#1\big]}

\newcommand{\paraa}[1]{\big(#1\big)}
\newcommand{\parab}[1]{\Big(#1\Big)}

\newtheorem{theorem}{Theorem}[section]
\newtheorem{corollary}[theorem]{Corollary}
\newtheorem{lemma}[theorem]{Lemma}
\newtheorem{proposition}[theorem]{Proposition}
\newtheorem{example}[theorem]{Example}
\theoremstyle{definition}
\newtheorem{definition}[theorem]{Definition}
\theoremstyle{remark}
\newtheorem{remark}[theorem]{Remark}
\numberwithin{equation}{section}

\newcommand{\K}{\mathbb K}

\newcommand{\A}{\mathcal{A}}

\newcommand{\triplebcdot}{[\cdot,\cdot,\cdot]}
\newcommand{\doublebcdot}{[\cdot,\cdot]}
\newcommand{\im}{\operatorname{im}}
\newcommand{\tr}{\operatorname{tr}}

\begin{document}

\title[Ternary Hom-Nambu-Lie algebras induced by Hom-Lie algebras]{Ternary Hom-Nambu-Lie algebras \\
  induced by Hom-Lie algebras}

\author{Joakim Arnlind}
\address{Max Planck Institute for Gravitational Physics (AEI), Am M\"uhlenberg 1, D-14476 Golm, Germany.}
\email{arnlind@aei.mpg.de}

\author{Abdenacer Makhlouf}
\address{Universit\'{e} de Haute Alsace,  Laboratoire de Math\'{e}matiques, Informatique et Applications,
4, rue des Fr\`eres Lumi\`ere F-68093 Mulhouse, France}
\email{abdenacer.makhlouf@uha.fr}

\author{Sergei Silvestrov}
\address{Centre for Mathematical Sciences,  Lund University, Box
   118, SE-221 00 Lund, Sweden}
\email{ssilvest@maths.lth.se}

\thanks {
This work was partially supported by the Crafoord Foundation, The
Swedish Foundation for International Cooperation in Research and
Higher Education (STINT), The Swedish Research Council, The Royal Swedish Academy of Sciences, The
Royal Physiographic Society in Lund, and The SIDA Foundation.
We also would like to thank the Royal Institute of Technology for
support and hospitality in connection to visits there while working on this project.}


\begin{abstract}
  The purpose of this paper is to investigate ternary multiplications
  constructed from a binary multiplication, linear twisting maps and a
  trace function. We provide a construction of ternary Hom-Nambu and
  Hom-Nambu-Lie algebras starting from a binary multiplication of a
  Hom-Lie algebra and a trace function satisfying certain
  compatibility conditions involving twisting maps.  We show that
  mutual position of kernels of twisting maps and the trace play
  important role in this context, and provide examples of
  Hom-Nambu-Lie algebras obtained using this construction.
\end{abstract}

\maketitle

\section{Introduction}\label{sec:intro}

\noindent The $n$-ary algebraic structures and in particular ternary algebraic
structures appear naturally in various domains of
theoretical and mathematical physics as for example Nambu  mechanics
\cite{n:generalizedmech,Takhtajan}. Further recent
motivation to study $n$-ary operations comes from string theory and
M-branes involving naturally an algebra with a ternary operation,
called Bagger-Lambert algebra \cite{BL2007}.    For other physical
applications see \cite{Kerner7,Kerner, Kerner2,Kerner4,Kerner6}. This provides additional motivation for   development of mathematical concepts such as Leibniz $n$-ary algebras \cite{CassasLodayPirashvili,f:nliealgebras}.

In \cite{ams:gennambu}, generalizations of $n$-ary algebras of
Lie type and associative type by twisting the identities using
linear maps have been introduced. These generalizations include $n$-ary Hom-algebra structures
generalizing the $n$-ary algebras of Lie type including $n$-ary
Nambu algebras, $n$-ary Nambu-Lie algebras and $n$-ary Lie algebras, and
$n$-ary algebras of associative type including $n$-ary totally
associative and $n$-ary partially associative algebras.

The general Hom-algebra structures arose first in connection to
quasi-deformation and discretizations of Lie algebras of vector
fields. These quasi-deformations lead to quasi-Lie algebras, a
generalized Lie algebra structure in which the skew-symmetry and
Jacobi conditions are twisted.  The first examples were concerned with
$q$-deformations of the Witt and Virasoro algebras (see for example
\cite{AizawaSaito,ChaiElinPop,ChaiKuLuk,ChaiPopPres,ChaiKuLukPopPresn,CurtrZachos1,
  DaskaloyannisGendefVir,Hu}).  Motivated by these and new examples
arising as applications of the general quasi-deformation construction
of \cite{HLS,LS1,LS2} on the one hand, and the desire to be able to
treat within the same framework such well-known generalizations of Lie
algebras as the color and super Lie algebras on the other hand,
quasi-Lie algebras and subclasses of quasi-Hom-Lie algebras and
Hom-Lie algebras were introduced in \cite{HLS,LS1,LS2,LS3}. In the
subclass of Hom-Lie algebras skew-symmetry is untwisted, whereas the
Jacobi identity is twisted by a single linear map and contains three
terms as for Lie algebras, reducing to ordinary Lie algebras when the
linear twisting map is the identity map.  Hom-associative algebras
replacing associative algebras in the context of Hom-Lie algebras and
also more general classes of Hom-Lie admissible algebras,
$G$-Hom-associative algebras, where introduced in \cite{MS}.  The
first steps in the construction of universal enveloping algebras for
Hom-Lie algebras have been made in \cite{Yau:EnvLieAlg}.  Formal
deformations and elements of (co-)homology for Hom-Lie algebras have
been studied in \cite{HomDeform,Yau:HomolHomLie}, whereas dual
structures such as Hom-coalgebras, Hom-bialgebras and Hom-Hopf
algebras appeared first in \cite{HomHopf,HomAlgHomCoalg} and further
investigated in \cite{IGSC,Yau:HomBial}.

This paper is organized as follows.  In Section \ref{sec:ternaryhom}
we review basic concepts of Hom-Lie, ternary Hom-Nambu and ternary
Hom-Nambu-Lie algebras. We also recall the method of composition with
endomorphism for the construction of Hom-Lie, ternary Hom-Nambu and
ternary Hom-Nambu-Lie algebras from Lie, Nambu and Nambu-Lie
algebras. In Section \ref{sec:inducednambu} we provide a construction
procedure of ternary Hom-Nambu and Hom-Nambu-Lie algebras starting
from a binary bracket of a Hom-Lie algebra and a trace function
satisfying certain compatibility conditions involving the twisting
maps. To this end, we use the ternary bracket introduced in
\cite{amdy:quantnambu}.  In Section \ref{sec:propertiestriple}, we
investigate how restrictive the compatibility conditions are. The
mutual position of kernels of twisting maps and the trace play an
important role in this context.  Finally, in Section
\ref{sec:examples}, we provide examples of Hom-Nambu-Lie algebras
obtained using constructions presented in the paper.

\section{Ternary Nambu-Lie  algebras}\label{sec:ternaryhom}

\noindent Let us first recall some basic facts about Hom-algebras. A
Hom-algebra structure is a multiplication on a vector space, which is
twisted by a linear map. In what follows, all vector spaces will be
defined over a field $\K$ of characteristic $0$, and $V$ will always
denote such a vector space.

The notion of a Hom-Lie algebra was initially motivated by examples of deformed Lie algebras coming from
twisted discretizations of vector fields \cite{HLS,LS1,LS2}.
We will follow notation conventions of \cite{MS}.

\begin{definition} \label{def:HomLie}
A \emph{Hom-Lie algebra} is a triple $(V, [\cdot, \cdot], \alpha)$
where $[\cdot, \cdot]: V\times V \rightarrow V$ is a bilinear map
and $\alpha: V \rightarrow V$
 a linear map
 satisfying
$$\begin{array}{c} [x,y]=-[y,x],  \quad {\text{(skew-symmetry)}} \\{}
\circlearrowleft_{x,y,z}{[\alpha(x),[y,z]]}=0 \quad
{\text{(Hom-Jacobi condition)}}
\end{array}$$
for all $x, y, z$ from $V$, where $\circlearrowleft_{x,y,z}$ denotes
summation over the cyclic permutations of $x,y,z$.
\end{definition}

\noindent Analogously, one can introduce Hom-algebra equivalents of
$n$-ary algebras \cite{ams:gennambu}.
\begin{definition}
  A \emph{ternary Hom-Nambu algebra} is a triple
  $\paraa{V,[\cdot,\cdot,\cdot],\tilde{\alpha}}$, consisting of a
  vector space $V$, a trilinear map $\triplebcdot:V\times V\times
  V\rightarrow V$ and a pair of linear maps
  $\tilde{\alpha}=(\alpha_1,\alpha_2)$ satisfying
  \begin{equation}\label{TernaryNambuIdentity}
    \begin{split}
      [ \alpha_1(x_{1}),&\alpha_{2}(x_{ 2}), [x_{3},x_{4},x_{5}]] =
      [ [x_1,x_2,x_3 ],\alpha_1(x_{4}),\alpha_{2}(x_{5})] \\
      &+ [ \alpha_{1} (x_3),[ x_1,x_2,x_{4}] , \alpha_{2}(x_{5})]+
      [\alpha_1(x_3),\alpha_{2}(x_{4}),[x_1,x_2,x_{5}]].
    \end{split}
  \end{equation}
\end{definition}
\noindent The identity \eqref{TernaryNambuIdentity} is called  ternary
Hom-Nambu identity.

\begin{remark}
Let $(V,[\cdot,\cdot,\cdot],\tilde{\alpha})$ be a ternary Hom-Nambu algebra
 where $\tilde{\alpha}=(\alpha_{1},\alpha_{2})$. Let $x=(x_1,x_{2})\in
V\times V$, $\tilde{\alpha}(x)=(\alpha_1(x_1),\alpha_{2}(x_{2}))\in V\times
V$ and $y\in V$. Let $L_x$ be a linear map on $V$,
 defined   by  $$L_{x}(y)=[x_{1},x_{2},y].$$
Then the Hom-Nambu identity  writes
\begin{align*}
  L_{\tilde{\alpha} (x)}( [x_{3},x_{4},x_{5}])= &[L_{x}(x_{3}),\alpha_1(x_{4}),
  \alpha_{2}(x_{5})]+[\alpha_1(x_{3}), L_{x}(x_{4}),
  \alpha_{2}(x_{5})]\\
  &+[\alpha_1(x_{3}), \alpha_{2}(x_{4}), L_{x}(x_{5})].
\end{align*}
\end{remark}

\begin{remark}
When the maps $(\alpha_{i})_{i=1,2}$ are all identity maps, one
recovers the classical ternary Nambu algebras. The identity obtained
in this special case of classical $n$-ary Nambu algebra is known
also as the fundamental identity or Filippov identity
\cite{f:nliealgebras, n:generalizedmech,Takhtajan}.
\end{remark}


\begin{definition} \label{def:hom-Nambu-Liealg}
  A ternary Hom-Nambu algebra
  $\paraa{V,\triplebcdot,(\alpha_1,\alpha_2)}$ is called
  a \emph{ternary Hom-Nambu-Lie algebra} if the bracket is
  skew-symmetric, that is
  \begin{equation} \label{trilskewsym}
  [x_{\sigma (1)}, x_{\sigma (2)}, x_{\sigma (3)}]=Sgn(\sigma )[x_{1}, x_{2}, x_{3}],\quad
    \forall\, \sigma \in\mathcal{S}_{3}\text{ and }\forall\, x_{1}, x_{2}, x_{3}\in V
  \end{equation}
  where $\mathcal{S}_{3}$ stands for the permutation group on $3$
  elements.
\end{definition}

\noindent The morphisms of ternary Hom-Nambu algebras  are defined in the
natural way. It should be pointed out however that the morphisms
should intertwine not only the ternary products but also the
twisting linear maps.
Let $(V,[\cdot,\cdot,\cdot],\tilde{\alpha})$ and
$(V',[\cdot,\cdot,\cdot]',\tilde{\alpha}')$ be two $n$-ary Hom-Nambu
algebras where $\tilde{\alpha}=(\alpha_{1},\alpha_{2})$ and
$\tilde{\alpha}'=(\alpha'_{1},\alpha'_{2})$. A linear map $\rho:
V\rightarrow V'$ is a ternary Hom-Nambu algebra morphism if it
satisfies
\begin{eqnarray*}\rho ([x_{1}, x_{2}, x_{3}])&=&
  [\rho (x_{1}),\rho (x_{2}),\rho (x_{3})]'\\
  \rho \circ \alpha_i&=&\alpha'_i\circ \rho \quad \text{ for } i=1,2.
\end{eqnarray*}
The following theorem, given in \cite{ams:gennambu} for $n$-ary algebras
of Lie type, provides a way to construct ternary Hom-Nambu algebras
(resp. $n$-ary Hom-Nambu-Lie algebras) starting from ternary Nambu
algebra (resp. $n$-ary Nambu-Lie algebra) and an algebra
endomorphism.
\begin{theorem}[\cite{ams:gennambu}]\label{ThmConstrHomNambuLie}
  Let $(V,[\cdot,\cdot,\cdot])$ be a ternary Nambu algebra (resp.
  ternary Nambu-Lie algebra) and let $\rho : V\rightarrow V$ be a
  ternary Nambu (resp. ternary Nambu-Lie) algebra endomorphism.  If we
  set $\tilde{\rho}=(\rho,\rho)$, then
  $(V,\rho\circ[\cdot,\cdot,\cdot],\tilde{\rho})$ is a ternary
  Hom-Nambu algebra (resp.  ternary Hom-Nambu-Lie algebra).

  Moreover, suppose that $(V',[\cdot,\cdot,\cdot]')$ is another
  ternary Nambu algebra (resp.  ternary Nambu-Lie algebra) and $\rho '
  : V'\rightarrow V'$ is a ternary Nambu (resp.  ternary Nambu-Lie)
  algebra endomorphism. If $f:V\rightarrow V'$ is a ternary Nambu
  algebra morphism (resp. ternary Nambu-Lie algebra morphism) that
  satisfies $f\circ\rho=\rho'\circ f$ then
  \begin{align*}
    f:(V,\rho\circ[\cdot,\cdot,\cdot],\tilde{\rho})\longrightarrow
    (V',\rho'\circ[\cdot,\cdot,\cdot]',\tilde{\rho}')
  \end{align*}
  is a ternary Hom-Nambu algebra morphism (resp.  ternary
  Hom-Nambu-Lie algebra morphism).
\end{theorem}

\begin{example}
  An algebra $V$ consisting of polynomials or possibly of other
  differentiable functions in $3$ variables $x_{1},x_{2},x_{3},$
  equipped with well-defined bracket multiplication given by the
  functional jacobian $J(f)=(\frac{\partial f_i}{\partial x_j})_{1\leq
    i,j \leq 3}$:
  \begin{equation} \label{JacternaryNambuLie} \lbrack
    f_{1},f_{2},f_{3}\rbrack=\det \left(
      \begin{array}{ccc}
        \frac{\partial f_{1}}{\partial x_{1}} & \frac{\partial
          f_{1}}{\partial x_{2}} &
        \frac{\partial f_{1}}{\partial x_{3}} \\
        \frac{\partial f_{2}}{\partial x_{1}} & \frac{\partial
          f_{2}}{\partial x_{2}} &
        \frac{\partial f_{2}}{\partial x_{3}} \\
        \frac{\partial f_{3}}{\partial x_{1}} & \frac{\partial
          f_{3}}{\partial x_{2}} &
        \frac{\partial f_{3}}{\partial x_{3}}%
      \end{array}%
    \right),
  \end{equation}%
  is a ternary Nambu-Lie algebra. By considering a ternary Nambu-Lie
  algebra endomorphism of such algebra, we construct a ternary
  Hom-Nambu-Lie algebra. Let $\gamma(x_1,x_2,x_3)$ be a polynomial or
  a more general differentiable transformation of three variables,
  mapping elements of $V$ to elements of $V$ under composition $f\mapsto f\circ \gamma $, and such that
  $\det J (\gamma) = 1$. Let $\rho_\gamma: V \mapsto V$ be the composition
  transformation defined by $f\mapsto f\circ \gamma $ for any $f \in
  V$. By the general chain rule for composition of transformations of
  several variables,
  \begin{eqnarray*}
    J(\rho_\gamma (f)) = J(f\circ \gamma)=(J(f)\circ \gamma) J(\gamma)=\rho_\gamma (J(f))J(\gamma), \\
    \det J(\rho_\gamma (f)) = \det (J(f)\circ \gamma) \det J(\gamma)=
    \det \rho_\gamma (J(f))\det J(\gamma).
  \end{eqnarray*}
  Hence, for any transformation $\gamma$ with $\det J(\gamma)=1$, the
  composition transformation $\rho_\gamma$ defines an endomorphism of
  the ternary Nambu-Lie algebra with ternary product
  \eqref{JacternaryNambuLie}. Therefore, by Theorem
  \ref{ThmConstrHomNambuLie}, for any such transformation $\gamma$,
  the triple
$$\left(V,\rho_\gamma\circ \lbrack \cdot,\cdot,\cdot \rbrack,(\rho_\gamma,\rho_\gamma)\right)$$
is a ternary Hom-Nambu-Lie algebra.
\end{example}

\section{Hom-Nambu-Lie algebras induced by Hom-Lie algebras}\label{sec:inducednambu}

\noindent In this Section we provide a construction procedure of
ternary Hom-Nambu and Hom-Nambu-Lie algebras starting from a binary
bracket of a Hom-Lie algebra and a trace function satisfying certain
compatibility conditions involving the twisting maps. To this end, we
use the ternary bracket introduced in \cite{amdy:quantnambu}.
\begin{definition} \label{def:bintracetril}
  Let $(V,[\cdot,\cdot])$ be a binary algebra and let $\tau: V
  \to \K$ be a linear map. The trilinear map
  $[\cdot,\cdot,\cdot]_{\tau}:V\times V\times V\to V$ is defined as
  \begin{align}\label{NambuCommutator1}
    [x,y,z]_{\tau}=\tau(x) [y,z]+ \tau(y) [z,x]+\tau(z) [x,y].
  \end{align}
\end{definition}

\begin{lemma}  \label{lem:skewsymbinskewsymtril}
If the bilinear multiplication $[\cdot,\cdot]$ in Definition \ref{def:bintracetril} is skew-symmetric, then
the trilinear map $[\cdot,\cdot,\cdot]_{\tau}$ is skew-symmetric as well.
\end{lemma}

\begin{proof}
The permutation group $S_3$ is generated by the two transpositions of neighboring indexes $(1,2)$ and $(2,3)$ and  $Sgn$ is multiplicative functional.
Thus the proof of skew-symmetry \eqref{trilskewsym} will be completed if it is done for $\sigma \in \{(1,2), (2,3)\}$.
This is proved using skew-symmetry of the bilinear $[\cdot,\cdot]$ multiplication as follows.
If  $\sigma = (1,2) $, then
\begin{eqnarray*}
[x_{\sigma (1)}, x_{\sigma (2)}, x_{\sigma (3)}] &=& [x_{2}, x_{1}, x_{3}]= \tau(x_2) [x_{1}, x_{3}] + \tau(x_1) [x_{3}, x_{2}] + \tau(x_3) [x_{2}, x_{1}]= \\
&=& -\tau(x_1) [x_{2}, x_{3}] - \tau(x_2) [x_{3}, x_{1}] - \tau(x_3) [x_{1}, x_{2}]= \\
&=& -[x_{1}, x_{2}, x_{3}]=Sgn(\sigma )[x_{1}, x_{2}, x_{3}].
\end{eqnarray*}
If  $\sigma = (2,3) $, then
\begin{eqnarray*}
[x_{\sigma (1)}, x_{\sigma (2)}, x_{\sigma (3)}] &=& [x_{1}, x_{3}, x_{2}]= \tau(x_1) [x_{3}, x_{2}] + \tau(x_3) [x_{2}, x_{1}] + \tau(x_2) [x_{1}, x_{3}]= \\
&=& -\tau(x_1) [x_{2}, x_{3}] - \tau(x_2) [x_{3}, x_{1}] - \tau(x_3) [x_{1}, x_{2}]= \\
&=& -[x_{1}, x_{2}, x_{3}]=Sgn(\sigma )[x_{1}, x_{2}, x_{3}].\qedhere
\end{eqnarray*}
\end{proof}

\noindent If $\tau:V\to\K$ is a linear map such
that $\tau([x,y])=0$ for
all $x,y\in V$, then we call $\tau$ a \emph{trace function on
  $(V,[\cdot,\cdot])$}. It follows immediately that $\tau([x,y,z]_{\tau})=0$ for all
$x,y,z\in V$ if $\tau$ is a trace function.

\begin{theorem}\label{thm:inducedHomLie}
  Let $(V,[\cdot,\cdot],\alpha)$ be a Hom-Lie algebra and
  $\beta:V\to V$ be a linear map. Furthermore, assume that $\tau$ is a
  trace function on $V$ fulfilling
  \begin{align}
    \tau\paraa{\alpha(x)}\tau(y) &=\tau(x)\tau\paraa{\alpha(y)}\label{eq:taualpha}\\
    \tau\paraa{\beta(x)}\tau(y) &=\tau(x)\tau\paraa{\beta(y)}\label{eq:taubeta}\\
    \tau\paraa{\alpha(x)}\beta(y) &=\tau\paraa{\beta(x)}\alpha(y)\label{eq:taualphabeta}
  \end{align}
  for all $x,y\in V$. Then
  $(V,[\cdot,\cdot,\cdot]_{\tau},(\alpha,\beta))$ is a Hom-Nambu-Lie
  algebra, and we say that it is induced by  $(V,[\cdot,\cdot],\alpha)$.
\end{theorem}

\begin{proof}
  Since $[\cdot,\cdot,\cdot]_\tau$ is skew-symmetric and trilinear by construction,
  one only has to prove that the Hom-Nambu identity is
  fulfilled. Expanding the Hom-Nambu identity
  \begin{align*}
    [\alpha(x),\beta(y),[z,u,v]] =
    [[&x,y,z],\alpha(u),\beta(v)]+
    [\alpha(z),[x,y,u],\beta(v)]\\
    &+[\alpha(z),\beta(u),[x,y,v]]
  \end{align*}
  gives 24 different terms. Six of these can be grouped into three pairs as follows
  \begin{align*}
    &\bracketa{\beta(v),[x,y]}\parab{\tau\paraa{\alpha(u)}\tau(z)-\tau\paraa{\alpha(z)}\tau(u)}\\
    &\bracketa{\tau\paraa{\alpha(z)}\tau(v)\beta(u)-\tau\paraa{\beta(v)}\tau(z)\alpha(u),[x,y]}\\
    &\bracketa{\alpha(z),[x,y]}\parab{\tau\paraa{\beta(v)}\tau(u)-\tau\paraa{\beta(u)}\tau(v)},
  \end{align*}
  which all vanish separately by \eqref{eq:taualpha},
 \eqref{eq:taubeta} and \eqref{eq:taualphabeta}. The remaining 18
  terms can be grouped in to six triples of the type
  \begin{align*}
    \tau\paraa{\alpha(x)}\tau(u)\bracketa{\beta(y),[z,v]}+
    \tau\paraa{\alpha(u)}\tau(x)\bracketa{\beta(v),[y,z]}+
    \tau\paraa{\beta(u)}\tau(x)\bracketa{\alpha(z),[v,y]}.
  \end{align*}
  By using \eqref{eq:taualphabeta} one can rewrite this term as
  \begin{align*}
    \tau\paraa{\beta(x)}\tau(u)\bracketa{\alpha(y),[z,v]}+
    \tau\paraa{\beta(u)}\tau(x)\bracketa{\alpha(v),[y,z]}+
    \tau\paraa{\beta(u)}\tau(x)\bracketa{\alpha(z),[v,y]},
  \end{align*}
  and by using \eqref{eq:taubeta} and the Hom-Jacobi identity one sees
  that this term vanishes. The remaining five triples of terms can be
  shown to vanish in an analogous way. Hence, the Hom-Nambu identity
  is satisfied.
\end{proof}

\begin{remark}
  If we choose $\beta=\alpha$ in Theorem \ref{thm:inducedHomLie},
 conditions \eqref{eq:taualpha} -- \eqref{eq:taualphabeta} reduce to the single relation
  \begin{align*}
    \tau\paraa{\alpha(x)}\tau(y) &=\tau(x)\tau\paraa{\alpha(y)}.
  \end{align*}
\end{remark}

\noindent Choosing $\alpha$ and $\beta$ to be identity maps in Theorem
\ref{thm:inducedHomLie} one obtains the result in \cite{amdy:quantnambu}.

\begin{corollary}
  Let $(V,[\cdot,\cdot])$ be a Lie algebra and $\tau:V\to\K$ be a
  trace function on $V$. Then $(V,[\cdot,\cdot,\cdot]_\tau)$ is a
  Nambu-Lie algebra.
\end{corollary}

\noindent Relation \eqref{eq:taualphabeta} effectively allows for the
interchange of $\alpha$ and $\beta$ in equations involving
$\tau$. Therefore, even though $\beta$ is only assumed to be a linear
map, relation \eqref{eq:taualphabeta} induces a Hom-Jacobi identity
for $(V,[\cdot,\cdot])$ with respect to $\beta$ in many cases.

\begin{proposition}
  Let $(V,[\cdot,\cdot],\alpha)$ be a Hom-Lie algebra and let
  $\tau:V\to\K$ and $\beta:V\to V$ be linear maps satisfying
  $\tau\paraa{\alpha(x)}\beta(y) =\tau\paraa{\beta(x)}\alpha(y)$ for
  all $x,y\in V$, and such that there exists an element $v\in V$ with
  $\tau\paraa{\alpha(v)}\neq 0$. Then $(V,[\cdot,\cdot],\beta)$ is a
  Hom-Lie algebra.
\end{proposition}

\begin{proof}
  Multiplying the Hom-Jacobi identity with $\tau\paraa{\beta(v)}$
  gives
  \begin{align*}
    \tau\paraa{\beta(v)}\parab{\bracketa{\alpha(x),[y,z]}+\bracketa{\alpha(y),[z,x]}+\bracketa{\alpha(z),[x,y]}}=0,
  \end{align*}
  and by using the relation between $\alpha$, $\beta$ and $\tau$ in all three terms one obtains
  \begin{align*}
    \tau\paraa{\alpha(v)}\parab{\bracketa{\beta(x),[y,z]}+\bracketa{\beta(y),[z,x]}+\bracketa{\beta(z),[x,y]}}=0.
  \end{align*}
  By assumption $\tau\paraa{\alpha(v)}\neq 0$, which reduces the above
  identity to the Hom-Jacobi identity for $\beta$. Since $x,y,z$ were
  chosen to be arbitrary, this proves that $(V,[\cdot,\cdot],\beta)$
  is a Hom-Lie algebra.
\end{proof}

\section{Properties of the compatibility conditions}\label{sec:propertiestriple}

\noindent When inducing a Hom-Nambu-Lie algebra via Theorem
\ref{thm:inducedHomLie}, one might ask how restrictive the
$\alpha,\beta,\tau$-compatibility conditions \eqref{eq:taualpha} --
\eqref{eq:taualphabeta} are? For instance, given a Hom-Lie algebra
$(V,[\cdot,\cdot],\alpha)$, how much freedom does one have to choose
$\beta$? It turns out that generically $\beta$ has to be
proportional to $\alpha$, except in the case when the images of
$\alpha$ and $\beta$ lie in the kernel of $\tau$ (see Proposition
\ref{prop:casesab}).

In the following we shall study consequences of the
$\alpha,\beta,\tau$-compatibility conditions by studying the kernels
of $\alpha$, $\beta$ and $\tau$.

\begin{definition}
  Let $V$ be a vector space, $\alpha$ and $\beta$ linear maps $V\to V$
  and $\tau$ a trace function on $V$. We say that the triple
  $(\alpha,\beta,\tau)$ is \emph{compatible on $V$} if conditions
  \eqref{eq:taualpha} -- \eqref{eq:taualphabeta} hold. Moreover, if
  $\ker\tau\neq \{0\}$ and $\ker\tau\neq V$ then we call the triple
  \emph{nondegenerate}.
\end{definition}

\begin{proposition}\label{prop:abelianKernel}
  Let $\A=(V,\triplebcdot_\tau,(\alpha,\beta))$ be a Hom-Nambu-Lie
  algebra induced by $(V,\doublebcdot,\alpha)$. If
  $\ker\tau=\{0\}$ or $\ker\tau=V$ then $\A$ is abelian.
\end{proposition}

\begin{proof}
  First, assume that $\ker\tau=\{0\}$. Since $\tau([x,y])=0$ for
  all $x,y\in V$ it follows that $[x,y]=0$ for all $x,y\in V$. By the
  definition of $\triplebcdot_\tau$ this implies that $[x,y,z]=0$ for all $x,y,z\in V$.

  Now, assume that $\ker\tau=V$. This directly implies (see
  \eqref{NambuCommutator1}) that $[x,y,z]=0$ for all $x,y,z\in V$.
\end{proof}

\noindent By an ideal of a Hom-Lie algebra $(V,[\cdot,\cdot],\alpha)$
we mean a subset $I\subseteq V$ such that $[I,V]\subseteq I$, and we
say that a Hom-Lie algebra is \emph{simple} if it has no ideals other
than $\{0\}$ and $V$. Since the kernel of a trace function is
always an ideal, one obtains the following corollary to Proposition
\ref{prop:abelianKernel}.

\begin{corollary}
  A Hom-Nambu-Lie algebra induced by a simple Hom-Lie
  algebra is abelian.
\end{corollary}

\begin{proof}
  Assume that the induced Hom-Nambu-Lie algebra is not abelian. Then,
  by Proposition \ref{prop:abelianKernel}, the kernel of $\tau$ is
  neither $\{0\}$ nor $V$. This implies that $\ker\tau$ is a
  non-trivial ideal of the Hom-Lie algebra, which contradicts that it
  is assumed to be simple.
\end{proof}

\noindent In Proposition \ref{prop:abelianKernel} we noted that if the
kernel of $\tau$ is either the complete vector space or $\{0\}$,
then the induced Hom-Nambu-Lie algebra will be abelian, and therefore
we shall focus on nondegenerate triples in the following. To fix
notation, we introduce $K=\ker\tau$ and $U$ such that $U=V\backslash
K$. For a nondegenerate compatible triple, $U$ and $K$ are always
different from $\{0\}$.

\begin{lemma}
  Let $(\alpha,\beta,\tau)$ be a nondegenerate compatible triple on
  $V$. Then it holds that $\alpha(K)\subseteq K$ and $\beta(K)\subseteq K$.
\end{lemma}

\begin{proof}
  Since the triple is assumed to be nondegenerate, one can find an
  element $y\in V$ such that $\tau(y)\neq 0$. For any $x\in K$,
  equations \eqref{eq:taualpha} and \eqref{eq:taubeta} imply that
  $\tau\paraa{\alpha(x)}=0$ and $\tau\paraa{\beta(x)}=0$. Hence,
  $\alpha$ and $\beta$ map $K$ into $K$.
\end{proof}

\begin{lemma}
  Let $(\alpha,\beta,\tau)$ be a nondegenerate compatible triple on
  $V$ and assume that there exists an element $u\in U$ such that
  $\alpha(u)\in K$ (or $\beta(u)\in K$). Then $\alpha(U)\subseteq K$
  (or $\beta(U)\subseteq K$).
\end{lemma}

\begin{proof}
  For a general $x\in V$ it follows from (\ref{eq:taualpha}) (on $u$
  and $x$) that $\tau(\alpha(x))=0$, since $u\in U$ and $\alpha(u)\in
  K$. An identical argument goes through for $\beta$ by using
  (\ref{eq:taubeta}).
\end{proof}

\noindent These results allow us to split the problem into four possible cases:
\begin{align}
  &\alpha(U)\subseteq U\text{ and }\beta(U)\subseteq U\tag{C1}\label{eq:C1}\\
  &\alpha(U)\subseteq K\text{ and }\beta(U)\subseteq K\tag{C2}\label{eq:C2}\\
  &\alpha(U)\subseteq U\text{ and }\beta(U)\subseteq K\tag{C3}\label{eq:C3}\\
  &\alpha(U)\subseteq K\text{ and }\beta(U)\subseteq U\tag{C4}\label{eq:C4}
\end{align}

\noindent Clearly, in case (\ref{eq:C2}) the compatibility conditions
will be identically satisfied since
$\tau\paraa{\alpha(x)}=\tau\paraa{\beta(x)}=0$ for all $x\in V$. In
the other cases, the next proposition shows that one does not have any
freedom to choose $\alpha$ and $\beta$ independently.

\begin{proposition}\label{prop:casesab}
  Let $(\alpha,\beta,\tau)$ be a nondegenerate compatible triple on
  $V$. Then, refering to the cases in (\ref{eq:C1}) -- (\ref{eq:C4}),
  the following holds:
  \begin{align}
    \exists\,\lambda\in&\K\backslash\{0\}:\,\, \beta=\lambda\alpha\tag{C1}\\
    \tag{C3}&\beta\equiv 0\\
    \tag{C4}&\alpha\equiv 0
  \end{align}
\end{proposition}

\begin{proof}
  Case (\ref{eq:C1}): If one chooses $u\in U$ and $x\in
  V$, relation (\ref{eq:taualphabeta}) gives
  \begin{align*}
    \beta(x)=\frac{\tau\paraa{\beta(u)}}{\tau\paraa{\alpha(u)}}\alpha(x),
  \end{align*}
  where $\tau\paraa{\beta(u)}\neq 0$ by assumption. Case
  (\ref{eq:C3}): By choosing $u\in U$ and $x\in V$, relation
  (\ref{eq:taualphabeta}) gives $\beta(x)=0$. Case (\ref{eq:C4}) is
  proven in the same way.
\end{proof}


\section{Examples}\label{sec:examples}
In this section we provide several examples of ternary Hom-Nambu-Lie algebras induced by Hom-Lie algebras by means of the method described in Theorem \ref{thm:inducedHomLie}.
From the results in the previous section (see Proposition \ref{prop:casesab}), there are two
(non-trivial) possibilities for $\alpha$ and $\beta$. Either
$\beta=\lambda\alpha$ or the images of $\alpha$ and $\beta$ are in the
kernel of $\tau$, in which case it is possible to have
$\beta\neq\lambda\alpha$. We provide examples in both cases.


\begin{example}
  In our first example, we let $V$ be the vector space of
  $n\times n$ matrices, and $\alpha:V\to V$ acts through conjugation by
  an invertible matrix $s$, i.e. $\alpha(x)=s^{-1}xs$. Then
  $(V,\alpha\circ\doublebcdot,\alpha)$ is a Hom-Lie algebra. For
  matrices, any trace function is proportional to the matrix trace, so
  we let $\tau(x)=\tr(x)$. If we want to choose a $\beta\neq 0$, it
  follows from Proposition \ref{prop:casesab} that $\beta$ has to be
  proportional to $\alpha$, i.e. $\beta=\lambda\alpha$ for some
  $\lambda\neq 0$. Since $\tr\paraa{\alpha(x)}=\tr(x)$ it is clear that
  $(\alpha,\lambda\alpha,\tr)$ is a nondegenerate compatible triple on
  $V$, which implies, by Theorem \ref{thm:inducedHomLie}, that
  $\paraa{V,\triplebcdot_{\tr},(\alpha,\lambda\alpha)}$ is a
  Hom-Nambu-Lie algebra induced by
  $(V,\alpha\circ\doublebcdot,\alpha)$.
\end{example}


\begin{example}
  Let us start with the vector space $V$ spanned by
  $\{x_1,x_2,x_3,x_4\}$ with a skew-symmetric bilinear map defined
  through
  \begin{align*}
    [x_i,x_j] = a_{ij}x_3+b_{ij}x_4
  \end{align*}
  where $a_{ij}$ and $b_{ij}$ are antisymmetric $4\times 4$ matrices. Defining
  \begin{align*}
    &\alpha(x_i) = x_3\qquad\beta(x_i)=x_4\qquad i=1,\ldots,4\\
    &\tau(x_1)=\gamma_1\qquad\tau(x_2)=\gamma_2\qquad\tau(x_3)=\tau(x_4)=0,
  \end{align*}
  one immediately observes that $\tau$ is a trace function,
  $\im\alpha\subseteq\ker\tau$, $\im\beta\subseteq\ker\tau$,
  and $\beta\neq\lambda\alpha$. Furthermore, $(V,[\cdot,\cdot],\alpha)$ is a
  Hom-Lie algebra provided
  \begin{align*}
    &b_{13}=b_{12}+b_{23}\\
    &b_{14}=b_{12}+b_{23}+b_{34}\\
    &b_{24}=b_{23}+b_{34}.
  \end{align*}
  The four independent ternary brackets of the induced Hom-Nambu-Lie algebra can be written as
  \begin{align*}
    &[x_1,x_2,x_3] = \paraa{\gamma_1a_{23}-\gamma_2a_{13}}x_3 + \paraa{\gamma_1 b_{23}-\gamma_2(b_{12}+b_{23})}x_4\\
    &[x_1,x_2,x_4] = \paraa{\gamma_1a_{24}-\gamma_2a_{14}}x_3 + \paraa{\gamma_1(b_{23}+b_{34})-\gamma_2(b_{12}+b_{23}+b_{34})}x_4\\
    &[x_1,x_3,x_4] = (\gamma_1a_{34})x_3 + (\gamma_1b_{34})x_4\\
    &[x_2,x_3,x_4] = (\gamma_2a_{34})x_3 + (\gamma_2b_{34})x_4.
  \end{align*}
  For instance, choosing $\gamma_1=\gamma_2=1$ and $a_{i<j}=1$,
  one obtains the Hom-Nambu-Lie algebra
  $\paraa{\langle x_1,x_2,x_3,x_4\rangle,[\cdot,\cdot,\cdot],(\alpha,\beta)}$ defined by \begin{align*}
    &[x_1,x_2,x_3] = -b_{12} x_4\\
    &[x_1,x_2,x_4] = -b_{34} x_4\\
    &[x_1,x_3,x_4] = x_3+b_{34} x_4\\
    &[x_2,x_3,x_4] = x_3+b_{34} x_4
  \end{align*}
  together with $\alpha(x_i)=x_3$ and $\beta(x_i)=x_4$.
\end{example}

\begin{example}
  We consider the 3-dimensional Hom-Lie algebra defined with respect to a basis $\{x_1,x_2,x_3\}$ by
  \begin{eqnarray*}
    \ [x_1,x_2]&=&a_1 x_2-\frac{a_2 a_4}{a_3} x_3, \\
    \ [x_1,x_3]&=&-\frac{a_1 a_3}{a_4} x_2+a_2 x_3,\\
    \ [x_2,x_3]&=&a_3 x_2+ a_4 x_3,
  \end{eqnarray*}
  where $a_1,a_2,a_3,a_4$ are parameters in $\K$ and $a_3,a_4\neq 0$.  The map
  $\alpha$ is defined by
  \begin{align*}
    &\alpha(x_1) = px_1\\
    &\alpha(x_2) = qx_3\\
    &\alpha(x_3) = qx_4,
  \end{align*}
  for any $p,q\in \K$.  We define a trace function as
  $$\tau (x_1)=t,\ \tau (x_2)=0,\ \tau (x_3)=0, $$ for any $t\in \K$.

  If $p\neq 0$, we let $\beta$ be the linear map defined by
  \begin{align*}
    &\beta(x_1) = rx_1\\
    &\beta(x_2) = \frac{qr}{p}x_2\\
    &\beta(x_3) = sx_3,
  \end{align*}
  for any $r,s\in \K$.  The previous data satisfies the
  conditions \eqref{eq:taualpha}, \eqref{eq:taubeta} and
  \eqref{eq:taualphabeta}. Then, according to Theorem
  \ref{thm:inducedHomLie}, we obtain a ternary Hom-Nambu-Lie algebra
  defined by
  $$ [x_1,x_2,x_3]= t(a_3x_2+a_4 x_3).$$

If $p=0$  then
  one may consider a map $\beta$ of the form
  \begin{align*}
    &\beta(x_1) = 0\\
    &\beta(x_2) = r_1x_1+r_2x_2+r_3x_3\\
    &\beta(x_3) = r_4x_1+r_5x_2+r_6x_3.
  \end{align*}
  for any $r_1 , r_2 , r_3, r_4 , r_5 , r_6\in \K$.
  The ternary bracket is the same as for $p\neq 0$ case.
\end{example}

\begin{example}
  We consider the 3-dimensional Hom-Lie algebra defined with respect to a basis $\{x_1,x_2,x_3\}$ by
  \begin{eqnarray*}
    \ [x_1,x_2]&=&-a_1 x_2+a_2  x_3, \\
    \ [x_1,x_3]&=& a_3 x_2+a_1 x_3,\\
    \ [x_2,x_3]&=&a_4 x_2+ a_5 x_3,
  \end{eqnarray*}
  where $a_1,a_2,a_3,a_4,a_5$ are parameters in $\K$.
  The map $\alpha$ is defined by
  \begin{align*}
    &\alpha(x_1) = 0\\
    &\alpha(x_2) = qx_3\\
    &\alpha(x_3) = qx_4,
  \end{align*}
  for any $q\in \K$, and we define a trace function as
  $$\tau (x_1)=t,\ \tau (x_2)=0,\ \tau (x_3)=0, $$ for any $t\in \K$.

  Let $\beta$ be the linear map defined by
  \begin{align*}
    &\beta(x_1) = 0\\
    &\beta(x_2) = r_1x_1+r_2x_2+r_3x_3\\
    &\beta(x_3) = r_4x_1+r_5x_2+r_6x_3,
  \end{align*}
  for any $r_1 , r_2 , r_3,r_4 r_5,r_6\in \K$.
  The previous data satisfies the condition \eqref{eq:taualpha},
  \eqref{eq:taubeta} and \eqref{eq:taualphabeta}. Then, according to
  Theorem \ref{thm:inducedHomLie}, we obtain a ternary Hom-Nambu-Lie
  algebra defined by
  $$ [x_1,x_2,x_3]= t(a_4x_2+a_5 x_3).$$
\end{example}

\end{document}